\documentclass[12pt]{amsart}

\usepackage{tgpagella}
\usepackage{euler}
\usepackage[T1]{fontenc}
\usepackage{amsmath, amssymb}
\usepackage[hidelinks]{hyperref}
\usepackage[english]{babel}
\usepackage{mathrsfs}
\usepackage{eucal}
\usepackage[all]{xy}
\usepackage{tikz}

\newtheorem{thm}{Theorem}[subsection]
\newtheorem*{thm*}{Theorem}
\newtheorem{lem}[thm]{Lemma}
\newtheorem*{prob*}{Problem}

\newtheorem{prop}[thm]{Proposition}
\newtheorem*{prop*}{Proposition}

\newtheorem{cor}[thm]{Corollary}
\newtheorem*{cor*}{Corollary}

\theoremstyle{definition}
\newtheorem{defn}[thm]{Definition}
\newtheorem*{defn*}{Definition}

\newtheorem{remark}[thm]{Remark}

\newtheorem*{question*}{Question}
\newtheorem*{Pquestion*}{Popa's question}
\newtheorem{conv}[thm]{Convention}
\newtheorem*{conv*}{Convention}

\makeatletter

\def\dotminussym#1#2{%
  \setbox0=\hbox{$\m@th#1-$}%
  \kern.5\wd0%
  \hbox to 0pt{\hss\hbox{$\m@th#1-$}\hss}%
  \raise.6\ht0\hbox to 0pt{\hss$\m@th#1.$\hss}%
  \kern.5\wd0}

\begin{document}

\title[Two new families of finitely generated simple groups of homeomorphisms]{Two new families of finitely generated simple groups of homeomorphisms of the real line}
\author{James Hyde}
\address{Malott Hall, Room 590, Cornell University, Ithaca NY}
\email{jameshydemaths@gmail.com}

\author{Yash Lodha}
\address{Center for mathematical challenges, Korea institute of advanced study, 85 Hoegi-ro, Cheongnyangri-dong, Dongdaemun-gu, Seoul}
\email{yashlodha763@gmail.com}
\urladdr{https://yl7639.wixsite.com/website}
\thanks{Yash Lodha is supported by the Samsung Science and Technology Foundation under Project Number SSTF-BA1301-51 and by a KIAS Individual Grant at the Korea Institute for Advanced Study}

\author{Crist\'{o}bal Rivas}
\address{Dpto. de Matem\'aticas y C.C., Universidad de Santiago de Chile}
\email{cristobal.rivas@usach.cl}

\maketitle

\begin{abstract}
The goal of this article is to exhibit two new families of finitely generated simple groups of homeomorphisms of $\mathbf{R}$.
These families are strikingly different from existing families owing to the nature of their actions on $\mathbf{R}$,
and exhibit surprising algebraic and dynamical features. In particular, one construction provides the first examples of finitely generated simple groups of homeomorphisms of the real line which also admit a minimal action by homeomorphisms on the circle. This provides new examples of finitely generated simple groups with infinite commutator width, and the first such left orderable examples.
Another construction provides the first examples of finitely generated simple left orderable groups that admit minimal actions by homeomorphisms on the torus.
\end{abstract}

\section{Introduction}

In \cite{HydeLodha} the first two authors constructed the first examples of finitely generated infinite simple groups of homeomorphisms of $\mathbf{R}$.
In a subsequent article \cite{MatteBonTriestino}, Matte Bon and Triestino demonstrated that certain groups of piecewise linear homeomorphisms of flows are also examples of this phenomenon.
Whether such groups exist had been a longstanding open question of Rhemtulla \cite{Martinez} 
(also asked by Clay and Rolfsen in \cite{ClayRolfsen}, by Navas in \cite{Navas2}, and in the ``Kourkova Notebook"  \cite{Kourkova}.)

The question as stated originally asks whether finitely generated simple left orderable groups exist. However, note that left orderability for countable groups is equivalent to requiring that they admit a faithful action by orientation preserving homeomorphisms of the real line.
Such constructions are difficult since achieving the combination of finite generation and simplicity presents certain technical challenges owing to the lack of compactness of $\mathbf{R}$. Moreover, there also certain natural obstructions to simplicity for various finitely generated groups of homeomorphisms of $\mathbf{R}$. If such a group is amenable, then it admits a homomorphism onto $\mathbf{Z}$ (see \cite{WitteMorris}). The same holds if the group admits a nontrivial action by $C^1$-diffeomorphisms on a closed interval (or even $[0,1)$, see \cite{Navas}).
For a more detailed discussion around these issues, we refer the reader to \cite{HydeLodha}.

The goal of this article is to exhibit two new families of examples that exhibit new, strikingly different dynamical and algebraic features, compared to existing families.
Groups in the first family are finitely generated by definition, however it is surprising that they are simple, and the proof of simplicity involves an intricate analysis of the group action.
The groups in the second family emerge as the derived subgroups of certain well known examples called \emph{fast $n$-ring groups} (defined independently by Brin, Bleak, Kassabov, Moore and Zaremsky in \cite{BBKMZ} and by the second author with Kim and Koberda in \cite{KKL}.). The simplicity of these examples is less surprising, however it is surprising that they are finitely generated and left orderable, the proof of finite generation involves an intricate analysis of the group action.

We now present the first family.
Recall that Thompson's group $T$ is the group of piecewise linear orientation preserving homeomorphisms of the circle $\mathbf{S}^1=\mathbf{R}/\mathbf{Z}$ such that:
\begin{enumerate}
\item Each linear part is of the form $2^n+d$ for $n\in \mathbf{Z},d\in \mathbf{Z}[\frac{1}{2}]/\mathbf{Z}$.
\item There are finitely many points where the slopes do not exist, and they lie in $\mathbf{Z}[\frac{1}{2}]$.
\end{enumerate} 

The group $\overline{T}<\textup{Homeo}^+(\mathbf{R})$ is the ``lift" of this action to the real line.
In particular, there is a short exact sequence $$1\to \mathbf{Z}\to \overline{T}\to T\to 1$$
Here the group $\mathbf{Z}$ is the group of integer translations of the real line, and it lies in the center of $\overline{T}$.
It is easily seen that $\overline{T}$ is finitely presented, since $T$ is finitely presented.
The group $\overline{T}$ was first studied by Ghys and Sergiescu in \cite{GhysSergiescu}, and it has several remarkable features.
For instance, it was observed recently by the first author with Belk and Matucci in \cite{BHM} that it contains an isomorphic copy of $\mathbf{Q}$,
and is the first example of an ``explicit" finitely presented group with this property.

One may modify the ``lift", $\overline{T}$, as follows.
Let $\mathbf{S}^1$ be as above, and consider the map $$\phi_{\lambda}:\mathbf{R}\to \mathbf{S}^1\qquad \mathbf{R}\to \mathbf{R}/\lambda \mathbf{Z}$$ for each $\lambda>0$.
The map $\phi_{\lambda}$ provides the alternative lift $\overline{T}_{\lambda}<\textup{Homeo}^+(\mathbf{R})$, which as an abstract group is isomorphic to $\overline{T}=\overline{T}_1$. Note that the center of $\overline{T}_{\lambda}$ is the group $\langle t\to t+n\lambda\mid n\in \mathbf{Z}\rangle$.
In spite of the fact that $\overline{T}, \overline{T}_{\lambda}$ are not simple, we prove the following:

\begin{thm}\label{main1}
Let $\lambda>1$ be irrational. The group $G_{\lambda}=\langle \overline{T}, \overline{T}_{\lambda}\rangle$ is simple.
\end{thm}

This provides a family of finitely generated simple subgroups of homeomorphisms of the real line which are very elementary to define.
Moreover, they are shown to admit minimal actions on the torus by homeomorphisms.

\begin{cor}\label{maincortorus}
There exist finitely generated simple groups of homeomorphisms of the real line that admit a minimal action by homeomorphisms on the torus.
\end{cor}

To describe the second family, we recall the notion of a \emph{fast $n$-ring group}.

\begin{defn}\label{nring}
Let $\{J_1,...,J_n\}$ be a set of connected open intervals in $\mathbf{S}^1$ that cover $\mathbf{S}^1$, and homeomorphisms $\{f_1,...,f_n\}$ that satisfy:
\begin{enumerate}
\item $J_i\cap J_j=\emptyset$ if $|i-j|>1(\text{ mod } n)$ and is a nonempty, proper, connected subinterval of both $J_i,J_j$ if $|i-j|=1(\text{ mod }n)$.
\item $J_i=Supp(f_i)=\{x\in \mathbf{S}^1\mid x\cdot f\neq x\}$ for each $1\leq i\leq n$.
\end{enumerate}
The aforementioned configuration is called an \emph{$n$-ring of intervals and homeomorphisms}.
The group $G_n=\langle f_1,...,f_n\rangle$ is said to be a \emph{fast $n$-ring group} if the following holds.
In what appears below, we interpret the subscripts as modulo $n$.
For each $1\leq i\leq n$, let $x_i$ be the endpoint of $J_{i+1}$ that lies in $J_i$. Then we have the following dynamical condition which we refer to throughout the article as $(*)$: 
$$x_i\cdot f_if_{i+1}...f_{i+l}\in J_{i+l+1}\qquad \forall 1\leq l\leq n$$
\end{defn}
It was demonstrated in \cite{BBKMZ} that the isomorphism type of $G_n$ does not depend on the choice of homeomorphisms $f_1,...,f_n$, provided the dynamical condition $(*)$ is satisfied.
Note that an elementary application of the classical ping pong lemma demonstrates that the group $G_2$ is the nonabelian free group of rank $2$.
The isomorphism type of $G_n$ for $n\geq 3$ remains mysterious, however.
Our second family emerges from the derived subgroups of these examples.

\begin{thm}\label{main2}
For each $n\geq 3$, the group $H_n=G_n'$ is finitely generated, simple and left orderable.
In fact, the prescribed action of $H_n$ on $\mathbf{S}^1$ lifts to a faithful action of $H_n$ on $\mathbf{R}$.
\end{thm}

To provide another dynamical motivation for this second family, we recall the following dynamical trichotomy for groups actions on the real line.
For every action of a finitely generated group $G$ by orientation preserving homeomorphisms of the real line without global fixed points, there are one of three possibilities:
\begin{enumerate}
\item[(i)] There is a $\sigma$-finite measure $\mu$ that is invariant under the action.
\item[(ii)] The action is semiconjugate to a minimal action for which every small enough interval is sent into a sequence of intervals that converge to a point under well chosen group elements, however, this
property does not hold for every bounded interval.
\item[(iii)] The action is globally contracting; more precisely, it is 
semiconjugate to a minimal one for which the contraction property above holds for all bounded intervals.
\end{enumerate}
(For details, we refer the reader to \cite{NavasICM}). 
Note that if a group admits a faithful action of type $(i)$, then it is \emph{indicable}: it admits a homomorphism onto $\mathbf{Z}$.
Therefore, finitely generated simple groups of homeomorphisms of the real line may only admit actions of type $(ii)$ or $(iii)$.
It was shown in \cite{HydeLodhaNavasRivas} that the groups $G_{\rho}$ constructed by the first two authors in \cite{HydeLodha} have the property that
every action on the real line by homeomorphisms without global fixed points is of type $(iii)$.
The same was shown by Matte Bon and Triestino for their examples in \cite{MatteBonTriestino}.
The following is a corollary of Theorem \ref{main2}, which illustrates a striking new phenomenon associated with the groups $H_n, n\geq 3$.

\begin{cor}\label{maincor}
There exist finitely generated simple left orderable groups which admit actions by orientation preserving homeomorphisms on the real line which are of type $(ii)$.
\end{cor}

Given a group that admits an action of type $(ii)$ on $\mathbf{R}$, it is easy to see that the action of the group on the orbit of $0$ provides an unbounded homogenous quasimorphism.
As a consequence, we have the following.

\begin{cor}\label{maincor1}
There exist finitely generated simple left orderable groups that admit non-trivial (unbounded) homogeneous quasimorphisms into the reals.
\end{cor}

This also has a nice algebraic consequence.
\begin{cor}\label{maincor2}
There exist finitely generated simple left orderable groups that have infinite commutator width: for each $n\in \mathbf{N}$ there is an element that cannot be expressed as a product of fewer than $n$ commutators.
\end{cor}

Note that the homeomorphisms that generate the group $G_n$ can be realised as elements of Thompson's group $T$.
It follows that the groups $H_n$ are subgroups of $T$. 
It was shown in \cite{GhysSergiescu} that the given action of $T$ on $\mathbf{S}^1$ is conjugate to an action on $\mathbf{S}^1$ by $C^{\infty}$-diffeomorphisms.

We conclude the following.

\begin{cor}
There exists a finitely generated simple left orderable group that admits:
\begin{enumerate}
\item a faithful action by $C^{\infty}$-diffeomorphisms of the circle.
\item a faithful action by piecewise linear homeomorphisms of the circle (with finitely many allowable breakpoints for each element).
\item an embedding into Thompson's group $T$.
\end{enumerate}
\end{cor}
Note that one may show $(1)$ directly without appealing to \cite{GhysSergiescu}, since we can simply choose the homeomorphisms that generate $G_n$ to be smooth.

\begin{conv}
In this article, all group actions will be right actions.
We will use the notation $[f,g]=fgf^{-1}g^{-1}$ and $f^g=g^{-1} fg$.
For $f\in \textup{Homeo}^+(\mathbf{R})$, we define $Supp(f)=\{x\in \mathbf{R}\mid x\cdot f\neq x\}$.
\end{conv}

\section{The first family}

The goal of this section is to prove Theorem \ref{main1}.
We first state and discuss a few preliminaries.

\subsection{Preliminaries}

Throughout the section we denote by $F$ the standard piecewise linear action of Thompson's group $F\leq \textup{Homeo}^+[0,1]$.
Also, we fix $\lambda\in \mathbf{R}\setminus \mathbf{Q}, \lambda>1$.
Recall that $G_{\lambda}=\langle \overline{T},\overline{T}_{\lambda}\rangle$
where $\overline{T}_{\lambda}<\textup{Homeo}^+(\mathbf{R})$ is the lift of the action of $T$ on $\mathbf{S}^1$ with the identification $\mathbf{R}\to \mathbf{R}/\lambda \mathbf{Z}$. 
We denote by $Z(\overline{T}),Z(\overline{T}_{\lambda})$ as the center of $\overline{T},\overline{T}_{\lambda}$ respectively.
Note that $$Z(\overline{T})=\langle t\to t+n\mid n\in \mathbf{Z}\rangle\qquad Z(\overline{T}_{\lambda})=\langle t\to t+n\lambda\mid n\in \mathbf{Z}\rangle$$

Recall that the pointwise stabilizer of $\mathbf{Z}$ in $\overline{T}$, which we denote by $F_1$, is naturally isomorphic to Thompson's group $F$.
Indeed the restriction of this action of $F_1$ to each interval $[n,n+1]$ for each $n\in \mathbf{Z}$ is conjugate to the standard piecewise linear action of $F$ on $[0,1]$ by the translation $t\to t+n$.
It is easy to see that for any element $f\in \overline{T}\setminus Z(\overline{T})$, $\langle\langle f\rangle\rangle=\overline{T}$.
Note that the analogous statement holds for $\overline{T}_{\lambda}$.



We observe the following.
\begin{lem}\label{uniformlycont}
Every $g\in G_{\lambda}$ is Lipshitz. In particular, $g$ is uniformly continuous. 
\end{lem}

\begin{proof}
It is straightforward to see that the elements of $\overline{T},\overline{T}_{\lambda}$ are Lipshitz.
Since this property for homeomorphisms is closed under composition and inverses, we are done.
\end{proof}


\subsection{The proof.}
The key idea in the proof of Theorem \ref{main1} is the following.

\begin{prop}\label{mainprop}
Let $f\in G_{\lambda}\setminus \{id\}$.
For each $c\in \{1,\lambda\}$, there is an element $g\in \langle \langle f \rangle \rangle$ that satisfies the following.
\begin{enumerate}
\item $x\cdot g=x$ for all $x\in c\cdot \mathbf{Z}$.
\item There exists a pair $x,y\in [0,c]$, $x<y$ such that $$(x+c\cdot n)\cdot g>y+c\cdot n\qquad \forall n\in \mathbf{Z}$$
\end{enumerate}
\end{prop}

Using Proposition \ref{mainprop}, we can finish the proof of Theorem \ref{main1} as follows.

\begin{proof}[Proof of Theorem \ref{main1}]
Let $g_1,g_2\in G_{\lambda}\setminus \{id\}$ be elements that satisfy the conclusion of Proposition \ref{mainprop} for $c=1,c=\lambda$, respectively.
We will show that:
\begin{enumerate}
\item $\langle \langle g_1\rangle \rangle_{G_{\lambda}}\cap (\overline{T}\setminus Z(\overline{T}))\neq \emptyset$.
\item  $\langle \langle g_2\rangle \rangle_{G_{\lambda}}\cap (\overline{T}_{\lambda}\setminus Z(\overline{T}_{\lambda}))\neq \emptyset$.
\end{enumerate}
We know that the normal closure of any element in $(\overline{T}\setminus Z(\overline{T}))$ is all of $\overline{T}$. Similarly, the normal closure of any element in $(\overline{T}_{\lambda}\setminus Z(\overline{T}_{\lambda}))$ is all of $\overline{T}_{\lambda}$.
So after showing the above we can conclude the proof.
Indeed, since the proofs for $(1),(2)$ are analogous, we shall just prove $(1)$.

Thanks to Lemma \ref{uniformlycont}, $g_1$ is Lipshitz. Combining this with the fact that there exists a pair $x,y\in [0,1]$, $x<y$ such that $$(x+n)\cdot g_1>y+n\qquad \forall n\in \mathbf{Z}$$
we obtain the following. There is an interval $I\subset (0,1)$ such that $x\in I$ and 
$$X\cdot g_1\cap X=\emptyset\qquad \text{ where } X=(\bigcup_{n\in \mathbf{Z}}(I+n))$$

We can find a nontrivial element $$h=[h_1,h_2]\qquad h,h_1,h_2\in F_1\leq \overline{T}$$ such that 
$$Supp(h_1),Supp(h_2),Supp(h)\subset X$$
It follows that $$[h_1,[h_2, g_1^{-1}]]=[h_1,h_2]=h\in F_1\leq (\overline{T}\setminus Z(\overline{T}))$$
finishing the proof.
\end{proof}

The rest of this section shall be devoted to proving Proposition \ref{mainprop}.
We will prove it for the case $c=1$, the other case is completely analogous.

\begin{defn}
Let $G<\textup{Homeo}^+(\mathbf{R})$ be a given group action.
Given compact intervals $I,J$ such that $|I|=|J|$, we denote by $T_{J,I}:\mathbf{R}\to \mathbf{R}$ the unique translation so that $T_{J,I}(J)=I$. 
Given $g,g_1,g_2\in G$ and $I,J$ as above, we define 
$$d_{g}(I,J)=sup\{(|x\cdot g-x\cdot h|\mid x\in I\}\qquad\text{ where } h=T_{I,J}\circ g\circ T_{J,I}$$
and 
$$d_{g_1,g_2}(I)=sup\{(|x\cdot g_1-x\cdot g_2|\mid x\in I\}$$

\end{defn}

\begin{lem}\label{lem uniformcons}
Consider an element $$g_1=u_1v_1...u_nv_n\in G_{\lambda}\qquad u_i\in \overline{T}, v_i\in \overline{T}_{\lambda}\text{ for each }1\leq i\leq n$$ 
For each $\epsilon>0$, there is a $\delta_1>0$ such that for each $\delta\in (-\delta_1,\delta_1)$, the element $$g_2=u_1(f_{\delta}^{-1}v_1f_{\delta})...u_n(f_{\delta}^{-1}v_nf_{\delta})\qquad \text{ where }x\cdot f_{\delta}=x+\delta$$
satisfies that $d_{g_1,g_2}([0,1])<\epsilon$.
\end{lem}

\begin{proof}
This follows from an elementary inductive argument on $n$, using continuity.
\end{proof}


The following is a basic dynamical fact about irrational translations,
and we leave the proof for the reader.

\begin{lem}\label{lem rept0}
Fix $\lambda\in \mathbf{R}\setminus \mathbf{Q},\lambda>1$.
For each $\epsilon>0$, there is an $N\in \mathbf{N}$ such that for any interval $I$ such that $|I|>N$, there are $m, k\in \mathbf{Z}$ such that $$[m,m+1]\subset I\qquad |m-k\lambda|<\epsilon$$
\end{lem}

\begin{defn}
An element $g\in \textup{Homeo}^+(\mathbf{R})$ is \emph{repetitive} if for each $\varepsilon>0$ 
there is an $N\in \mathbf{N}$ such that for each interval $I$ such that $|I|>N$, there is a subinterval $[m,m+1]\subset I, m\in \mathbf{Z}$ such that $$d_g([0,1],[m,m+1])<\epsilon$$
We say that a group action $G\leq \textup{Homeo}^+(\mathbf{R})$ is said to be \emph{repetitive}, if every $g\in G$ is repetitive.
\end{defn}

\begin{prop}\label{prop repet} 
$G_{\lambda}$ is repetitive.
\end{prop}

\begin{proof}
Consider a nontrivial element $g=u_1v_1\ldots u_nv_n\in G_{\lambda}$, where $u_i\in \overline{T}$ and $v_i\in \overline{T}_\lambda$. 
We will show that $g$ is repetitive.
Let $\epsilon>0$.
Applying Lemma \ref{lem uniformcons}, there is a $\delta_1>0$ such that for each $\delta\in (-\delta_1,\delta_1)$, the element $$g_{\delta}=u_1(t_{\delta}^{-1}v_1t_{\delta})...u_n(t_{\delta}^{-1}v_nt_{\delta})$$
satisfies that $d_{g,g_{\delta}}([0,1])<\epsilon$. 

Using Lemma \ref{lem rept0} we find an $N\in \mathbf{N}$ such that in every interval $I$ of length at least $N$ there are $m,k\in \mathbb{Z}$ such that $[m,m+1]\subset I$ and $|m-k\lambda|< \delta_1$. 
For such $m,k$ there is a  $\delta\in (-\delta_1,\delta_1)$ such that 
$$g\restriction [m,m+1]=f_m^{-1} g_{\delta} f_m\qquad \text{ where }t\cdot f_m=t+m$$

Combining this with the fact that $d_{g,g_{\delta}}([0,1])<\epsilon$, we obtain 
$$d_{g}([0,1],[m,m+1])<\epsilon$$

\end{proof}


For the rest of the section, we denote $\langle \langle g\rangle \rangle_{G_{\lambda}}$ as simply $\langle \langle g\rangle \rangle$.
We shall now focus our attention on the pointwise stabilizer of $\mathbf{Z}$ in $G_{\lambda}$.
Recall that the pointwise stabilizer of $\mathbf{Z}$ in $\overline{T}$ is $F_1$, defined in the preliminaries above.

\begin{prop}\label{prop mid}
Let $g\in G_{\lambda}\setminus \{id\}$. There is an open interval $J\subset (0,1)$ whose closure is also contained in $(0,1)$, and a nontrivial element $f\in \langle \langle g\rangle \rangle$ such that:
$$Supp(f)\subset \bigcup_{n\in \mathbf{Z}} (n+J)$$
\end{prop}

\begin{proof}
First we argue that $\langle \langle g\rangle \rangle$ must contain elements that do not lie in $\langle t\to t+n\mid n\in \mathbf{Z}\rangle$.
If $g$ is itself not in this subgroup, then we are done. Otherwise, we can find an element $f\in \overline{T}_{\lambda}$ such that the element $[g,f]\neq id$ and satisfies the required property. 

We assume in the rest of the proof that $g\notin \langle t\to t+n\mid n\in \mathbf{Z}\rangle$. It follows that there is an $m\in \mathbf{Z}$ and an open interval $J$ such that $\overline{J}\subset (m,m+1)$ and $g\restriction J$ is not the restriction of an integer translation. 

We choose a open subinterval $I_1\subset J$ such that for some $\epsilon>0$ it holds that
$$|inf(I_1)-inf(J)|, |sup(J)-sup(I_1)|>\epsilon$$
Since $g$ is Lipshitz, for $\epsilon>0$, there is a $\delta>0$ such that given an interval $I$, if $|I|<\delta$, then
$|I\cdot g|<\epsilon$.
We then choose an open interval $I\subset I_1$ such that $|I|<\delta$ and 
$$(I\cdot g)\cap \bigcup_{n\in \mathbf{Z}} (n+I)=\emptyset$$
From our assumption, we know that for any $n_1,n_2\in \mathbf{Z}$, if $(n_1+I)\cdot g\cap (n_2+I)\neq \emptyset$, then $(n_1+I)\cdot g\subset (n_2+J)$.

Let $f=[f_1,f_2]\in F_1$ be such that $$Supp(f), Supp(f_1),Supp(f_2)\subset \bigcup_{n\in \mathbf{Z}}(n+I)$$
Then the element $h=[f_1,[f_2,g^{-1}]]$ is the required element that satisfies the conclusion of the Proposition.
That is,
$$Supp(h)\subset \bigcup_{n\in \mathbf{Z}} (n+J)$$
\end{proof}

We define a map $\nu:F\to \overline{T}$ as the obvious extension of the natural map $\nu:F\to F_1\leq \overline{T}$.
For an open interval $I\subset (0,1)$ and $N\in \mathbf{N}$, an element $f\in G_{\lambda}$ is said to be \emph{$(I,N)$-stable}, if the following holds.
\begin{enumerate}

\item $Supp(f)\subset \bigcup_{n\in \mathbf{Z}}(I+n)$.

\item There is an element $g\in F\setminus \{id\}$ such that for each each interval $I,|I|>N$, there is an interval $[m,m+1]\subset I$ for $m\in \mathbf{Z}$,
such that $\nu(g)\restriction [m,m+1]=f\restriction [m,m+1]$.

\end{enumerate}

Note that the element $f$ that emerges in the conclusion of Proposition \ref{prop mid} satisfies the first of the two conditions above.
We show the following.

\begin{prop}\label{prop regular}
Let $I\subset (0,1)$ be an open interval whose closure is also contained in $(0,1)$, and let $f\in G_{\lambda}\setminus \{id\}$ such that:
$$Supp(f)\subset \bigcup_{n\in \mathbf{Z}} (n+I)$$
Then there is a nontrivial element $h\in \langle \langle f\rangle \rangle$ and an $N\in \mathbf{N}$ such that $h$ is $(I,N)$-regular.
\end{prop}

\begin{proof}
Assume without loss of generality that $f\restriction [0,1]$ is nontrivial.
(Otherwise, we may replace $f$ with a conjugate of $f$ by an integer translation and proceed.)
Let $J\subset I\subset (0,1)$ be an open interval such that either $sup(J\cdot f)<inf(J)$ or $sup(J)<inf(J\cdot f)$.
It follows that there is an $\epsilon>0$ such that for any $g\in \textup{Homeo}^+[0,1]$ such that $d_{g,f}([0,1])<\epsilon$, we have that $J\cdot g\cap J=\emptyset$.

Since $f$ is repetitive, there is an $N\in \mathbf{N}$ such that for each interval $I, |I|>N$ there is an interval $[k,k+1]\subset I$ for $k\in \mathbf{Z}$, such that $$d_f([0,1],[k,k+1])<\epsilon$$

Let $g=[g_1,g_2], g_1,g_2\in F$ be non trivial elements such that $$Supp(g),Supp(g_1),Supp(g_2)\subset J$$
Then the element $$h=[\nu(g_1),[\nu(g_2),f^{-1}]]\in \langle \langle f\rangle \rangle$$ has the property that
for each interval $I, |I|>N$ there is an interval $[k,k+1]\subset I$ for $k\in \mathbf{Z}$
such that $$\nu(g)\restriction [k,k+1]=h\restriction [k,k+1]$$
Moreover, $$Supp(h)\subset \bigcup_{k\in \mathbf{Z}}(I+k)$$
This finishes the proof.
\end{proof}

An element $f\in \textup{Homeo}^+(\mathbf{R})$ is said to be \emph{$\epsilon$-advancing} for some $\epsilon>0$, if for each $x\in \mathbf{R}$ we have that $x\cdot f\geq x-\epsilon$.
(The analagous notion is defined for a homeomorphism of a compact interval).

\begin{lem}\label{lem surgery}
Let $\alpha\in \textup{Homeo}^+[0,1]\setminus \{id\}$ such that $Supp(\alpha)\subseteq I$ for an open interval $I\subset (0,1)$ whose closure is also contained in $(0,1)$.
For any $\epsilon>0$ and $x,y\in (0,1), x<y$, there exist elements $g_1,...,g_n\in F, n\in \mathbf{N}$ such that: 
\begin{enumerate}
\item The element $h=(\alpha^{g_1})^{l_1}...(\alpha^{g_n})^{l_n}$ satisfies that $x\cdot h>y$ for some $l_1,...,l_n\in \mathbf{Z}$.
\item For any homeomorphism $\beta\in \textup{Homeo}^+[0,1], Supp(\beta)\subseteq I$ (where $I$ is the same interval as above),
we have that $(\beta^{g_1})^{l_1}...(\beta^{g_n})^{l_n}$ is $\epsilon$-advancing for any $l_1,...,l_n\in \mathbf{Z}$.
\end{enumerate}
\end{lem}

\begin{proof}
Recall that the action of $F$ on $(0,1)$ has the following two features, which we shall use.
The first is that for any pair of closed intervals $I_1,I_2\subset (0,1)$ with dyadic rational endpoints,
there is an element $f\in F$ such that $I_1\cdot f=I_2$.
The second is that for any triple of closed intervals $I_1\subset I_2\subset I_3\subset (0,1)$ with dyadic rationals as endpoints,
we can find an $f\in F$ such that $Supp(f)\subset I_3$ and that $I_2\subset I_1\cdot f$.
We use these features to construct elements $g_1,...,g_n\in F$ such that 
$$J_1=I\cdot g_1\qquad ...\qquad J_n=I\cdot g_n$$ are intervals satisfying:
\begin{enumerate}
\item For each $1\leq i\leq n-1$, we have 
$$inf(J_i)<inf(J_{i+1})<sup(J_i)<sup(J_{i+1})\qquad sup(J_i)<inf(J_{i+2})\text{ (for }i<n-1)$$

\item For each $1\leq i\leq n$, $|J_i|<\epsilon$.

\item $x\in J_1, y\in J_n$.

\end{enumerate}

Moreover, we choose $g_1,...,g_n\in F$ that also satisfy that for some $l_1,...,l_n\in \mathbf{Z}$ we have 
$$x\cdot (\alpha^{g_1})^{l_1}...(\alpha^{g_n})^{l_n}>y$$
Note that condition $(2)$ above guarantees that for any $\beta\in \textup{Homeo}^+[0,1]$ such that $Supp(\beta)\subseteq I$,
we have that $(\beta^{g_1})^{l_1}...(\beta^{g_n})^{l_n}$ is $\epsilon$-advancing for any $l_1,...,l_n\in \mathbf{Z}$.
\end{proof}

\begin{proof}[Proof of Proposition \ref{mainprop}]
Let $g\in G_{\lambda}\setminus \{id\}$. By combining Propositions \ref{prop mid} and \ref{prop regular}, we obtain an element $h\in \langle \langle g\rangle \rangle$, an open interval $I\subset (0,1)$ whose closure is also contained in $(0,1)$, and an $N\in \mathbf{N}$ such that $h$ is $(I,N)$-regular.
Replacing $h$ by a conjugate of an integer translation if necessary, we assume that for $[0,N]$, the interval $[0,1]$  
realizes condition $(2)$ of the definition of $(I,N)$-regular.

We apply Lemma \ref{lem surgery} to $\alpha=h\restriction [0,1]$ for $\epsilon=\frac{1}{8N}, x=\frac{1}{8}, y=\frac{3}{4}$ to obtain elements $g_1,...,g_n\in F, n\in \mathbf{N}$ such that: 
\begin{enumerate}
\item The element $\gamma=(\alpha^{g_1})^{l_1}...(\alpha^{g_n})^{l_n}$ satisfies that $\frac{1}{8}\cdot \gamma>\frac{3}{4}$ for some $l_1,...,l_n\in \mathbf{Z}$.
\item For any homeomorphism $\beta\in \textup{Homeo}^+[0,1], Supp(\beta)\subseteq I$ (where $I$ is the same interval as above),
we have that $(\beta^{g_1})^{l_1}...(\beta^{g_n})^{l_n}$ is $\epsilon$-advancing for any $l_1,...,l_n\in \mathbf{Z}$.
\end{enumerate}

Let $$\zeta_1=(h^{\nu(g_1)})^{l_1}...(h^{\nu(g_n)})^{l_n}$$
and $$\zeta_2=\prod_{0\leq k\leq N,k\in \mathbf{Z}} f_k^{-1}\zeta_1f_n\qquad \text{ where }t\cdot f_k=k+1$$
Then $\zeta_2$ is the required element that satisfies the conditions of Proposition \ref{mainprop} for $c=1$, that is:
\begin{enumerate}
\item $x\cdot \zeta_2=x$ for all $x\in \mathbf{Z}$.
\item And $$(\frac{1}{4}+n)\cdot \zeta_2>\frac{1}{2}+n\qquad \forall n\in \mathbf{Z}$$
\end{enumerate}

\end{proof}

\begin{proof}[Proof of Corollary \ref{maincortorus}]
We shall define an action of $G_{\lambda}$ on $\mathbf{R}^2$ that commutes with the natural action of $\mathbf{Z}^2$ on $\mathbf{R}^2$,
and preserves a lamination on $\mathbf{R}^2$.
We define homomorphisms $\eta_1,\eta_2:\overline{T}\to \textup{Homeo}^+(\mathbf{R}^2)$ as follows.
$$(x,y)\cdot \eta_1(g)=(x+\frac{1}{\lambda}(y\cdot g-y), y\cdot g)$$
$$(x,y)\cdot \eta_2(g)=(x\cdot g, y+\lambda(x\cdot g-x))$$
Note that here $x\cdot g, y\cdot g$ refers to the action of $\overline{T}$ on $\mathbf{R}$.
The group generated by $\eta_1(\overline{T}),\eta_2(\overline{T})$ provides a homomorphism $\eta:\overline{T}*\overline{T}\to \textup{Homeo}^+(\mathbf{R}^2)$ whose image is $G_{\lambda}$.
It is an easy exercise to check that this preserves the foliation obtained by lines that have an angle $\theta$ with the $x$-axis,
where $tan(\theta)=\lambda$, and the action on the leaves is conjugate to the original action on $\mathbf{R}$.
\end{proof}

\section{The second family}

Our goal in this section will be to prove Theorem \ref{main2}.

\subsection{Preliminaries}

We recall from the introduction the $n$-ring configuration of intervals $\{J_1,...,J_n\}$ and homeomorphisms $\{f_1,...,f_n\}$ that satisfy the dynamical condition $(*)$. 
Note that if the above is satisfied for some $1\leq i\leq n$, then it is satisfied for all $1\leq i\leq n$. 
As before, we denote the resulting group, called the fast $n$-ring group, as $G_n=\langle f_1,...,f_n\rangle$.
The following was proved in \cite{BBKMZ}.

\begin{thm}\label{fast}
Given an $n$-ring configuration of intervals and homeomorphisms that satisfies condition $(*)$,
the (marked) isomorphism type of the group $G_n$ (with generating set $\{f_1,...,f_n\}$) does not depend on the choice of homeomorphisms $f_1,...,f_n$.
\end{thm}

First, recall the following well known result (Theorem $2.1.1$ in \cite{Navas}).

\begin{thm}\label{minimalisation}
If $G<\textup{Homeo}^+(\mathbf{S}^1)$ then precisely one of the following holds:
\begin{enumerate}
\item There is a finite orbit.
\item All the orbits are dense.
\item There exists a copy of the cantor set $C\subset \mathbf{S}^1$, which is $G$-invariant, and such that $G\restriction C$ is minimal.
In this case, the given action is semiconjugate to a minimal action, i.e. there is a degree one continuous map $\Phi:\mathbf{S}^1\to \mathbf{S}^1$ 
and a group homomorphism $\psi:G\to H$ onto some group $H$ such that
$$\forall f\in G\qquad \Phi\circ f=\psi(f)\circ \Phi$$
The resulting minimal action of $H$ on $\mathbf{S}^1$ is called the minimalisation of the action of $G$.
\end{enumerate}
\end{thm}

\begin{remark}
Note that in case $3$, the minimalisation is an action of $H$, but since $H$ is a quotient of $G$ (possibly with trivial kernel),
it may be viewed as an action of $G$. In some cases (for instance if $G$ is the fast $n$-ring group), one can show that $\phi:G\to H$ must be an isomorphism.
\end{remark}

Given a group of homeomorphisms of the circle, we say that the action is \emph{proximal}, if for every interval $I\subset \mathbf{S}^1$ such that $\mathbf{S}^1\setminus I$ has nonempty interior, and any open set $J\subset \mathbf{S}^1$, there is an element $f$ in the group such that $I\cdot f\subset J$.

\subsection{The proof of Theorem \ref{main2}}

Observe that while a given action of $G_n$ may not be minimal, by part $(3)$ of Theorem \ref{minimalisation}, it is semiconjugate to a minimal action on $\mathbf{S}^1$. (Clearly the group action has no finite orbit.)
It is clear that the dynamical condition $(*)$ holds for the new minimal action as well.
Since this dynamical condition guarantees a stable isomorphism type (by Theorem \ref{fast}), it follows that this new minimal action of $G_n$ is also faithful.
Actually, the main Theorem of \cite{BBKMZ} in fact guarantees that we can choose the homeomorphisms $f_1,...,f_n$ satisfying the dynamical condition $(*)$ such that the action of $G_n$ on $\mathbf{S}^1$ is minimal. 
Therefore, for the rest of this section we shall assume that the action of $G_n$ on $\mathbf{S}^1$ is minimal.
We denote as before $H_n=G_n'$ and assume that $n\geq 3$.

\begin{lem}\label{proximal}
The action of $H_n$ on $\mathbf{S}^1$ is proximal. 
\end{lem}

\begin{proof}
First we shall prove that the action of $G_n$ on $\mathbf{S}^1$ is proximal. Let $I\subset \mathbf{S}^1$ such that $\mathbf{S}^1\setminus I$ has nonempty interior, and consider an open set $J\subset \mathbf{S}^1$.

By minimality, we find an element $g_1\in G_n$ such that $inf(J_1)\cdot g_1\subset J$.
By continuity, there is an open interval $I_1$ containing $inf(J_1)$ such that $I_1\cdot g_1\subset J$.
It is an elementary exercise using the definition of $f_1,...,f_n$, and minimality, to construct an element $g_2\in G_n$ such that $I\cdot g_2\subset Supp(f_1)$. 
There is an $n\in \mathbf{Z}$ such that $I\cdot g_2\cdot f_1^n\subset I_1$. It follows that $I\cdot g_2f_1^ng_1\subset J$, finishing the proof.

First we show that the action of $H_n$ is minimal.
Using proximality of $G_n$ as above, let $g_i\in G_n$ be an element such that $J_i\cdot g_i\subset J_i^c$.
We define the elements $$l_i=g_i^{-1}f_i^{-1}g_if_i$$
Clearly, $l_i\in H_n$ and $l_i\restriction Supp(f_i)=f_i\restriction Supp(f_i)$.
It follows that the orbits of the actions of $H_n, G_n$ are the same, hence the action of $H_n$ is minimal.
We show that the action of $H_n$ is proximal in a similar way as done above for $G_n$, replacing the elements $f_i$ by $l_i$.
(Note that one may construct $l_i$ above so that $Supp(l_i)\cap J_i^c$ lies in any given open interval $J\subset J_i^c$.)
\end{proof}

\begin{prop}
$H_n$ is simple.
\end{prop}

\begin{proof}
To prove simplicity, we must show that $\langle \langle g\rangle \rangle_{H_n}=H_n$ for an arbitrary $g\in H_n\setminus \{id\}$.
First we show that $\langle \langle g\rangle \rangle_{H_n}$ contains a nontrivial element $f$ such that $Supp(f)^c$ has nonempty interior.
Let $J$ be an open interval such that $J\cap (J\cdot g)=\emptyset$ and $(J\cup (J\cdot g))^c$ has nonempty interior.
It is easy to find an element $\gamma\in H_n$ such that $Supp(\gamma)^c$ has nonempty interior.
(Take $[f_1,f_2]$, for instance.)
Using proximality, we find an element $h_1\in H_n$ such that $Supp(\gamma)\cdot h_1\subset J$.
The element $f=[g^{-1}, \gamma^{h_1}]$
has the feature that $f\in \langle \langle g\rangle \rangle_{H_n}\setminus \{id\}$ and that $Supp(f)^c$ has nonempty interior.

Since $\langle \langle f\rangle \rangle_{H_n}\leq \langle \langle g\rangle \rangle_{H_n}$, showing that $\langle \langle f\rangle \rangle_{H_n}=H_n$ finishes the proof. 
It suffices to show that $[f_i,f_j]^h=[f_i^h,f_j^h]\in \langle \langle f\rangle \rangle_{H_n}$ for each $1\leq i,j\leq n$ and $h\in G_n$.

We denote $\beta_1=f_i^h, \beta_2=f_j^h$
and $K_1=Supp(\beta_1), K_2=Supp(\beta_2)$.
It is easy to see that $(K_1\cup K_2)^c$ has nonempty interior since $Supp([f_i,f_j])^c$ has the same feature.

Let $I_1,I_2,I_3$ be disjoint open intervals such that $$I_1\cdot f\cap I_1=\emptyset\qquad I_2\cup I_3\subset Supp(f)^c$$

Using proximality, we find $g_1,g_2,g_3\in H_n$ such that $$(K_1\cup K_2)\cdot g_1\subset I\qquad K_1\cdot g_2\subset I_2\qquad K_2\cdot g_3\subset I_3$$

Let $$\alpha_1=[g_1^{-1}, \beta_1][\beta_1,g_2^{-1}]=\beta_1^{g_1} (\beta_1^{-1})^{g_2}\qquad \alpha_2=[g_1^{-1}, \beta_2][\beta_2,g_3^{-1}]=\beta_2^{g_1} (\beta_2^{-1})^{g_3}$$

Note that $\alpha_1,\alpha_2\in H_n$.

We obtain $$[\beta_1,\beta_2]=[\alpha_1, [\alpha_2,f^{-1}]]^{g_1^{-1}}\in \langle \langle f\rangle \rangle_{H_n}$$
This proves our claim.
\end{proof}

\begin{defn}\label{Li}
We say that an open interval $I\subset \mathbf{S}^1$ is \emph{small} if there exists $1\leq k\leq n$ such that $I\subset J_k$.
We find a collection of pairwise disjoint small intervals $$\mathcal{I}=\{L_{i,j}\subset \mathbf{S}^1\mid 1\leq i,j\leq n\}$$
satisfying that $$L_{i,j}\subset int(\mathbf{S}^1\setminus J_j)\qquad 1\leq i,j\leq n$$
We say that a small open interval $I\subset \mathbf{S}^1$ is \emph{$\mathcal{I}$-small} if for each $1\leq i,j\leq n$ the following holds:
\begin{enumerate}
\item If $I\cap L_{i,j}\neq \emptyset$ then $I\cap J_j=\emptyset$.
\item If $I\cap J_j\neq \emptyset$, then $I\cap L_{i,j}=\emptyset$.
\end{enumerate}
It is an easy exercise to show that for any such $\mathcal{I}$,
there is an $\epsilon>0$ such that any interval $I$ with $|I|<\epsilon$ is $\mathcal{I}$-small.
More generally, we say that a subset of $\mathbf{S}^1$ is $\mathcal{I}$-small, if it is contained in a $\mathcal{I}$-small interval.

Using proximality from Lemma \ref{proximal}, we construct a set of $n^2$ elements $$\{\lambda_{i,j}\mid 1\leq i,j\leq n\}\subset G_n$$ 
satisfying that $$J_i\cdot \lambda_{i,j}\subset L_{i,j}\qquad \forall 1\leq i,j\leq n$$
We define $\nu_{i,j}=\lambda_{i,j}^{-1}f_i\lambda_{i,j}$.
Note that $Supp(\nu_{i,j})\subseteq L_{i,j}$.
Since the intervals $\{L_{i,j}\mid 1\leq i,j\leq n\}$ are pairwise disjoint, the elements of the set $\{\nu_{i,j}\mid 1\leq i,j\leq n\}$
generate a free abelian group of rank $n^2$.

We define the set $$X=\{\nu_{i,j}^{-1} f_i\mid 1\leq i,j\leq n\}\subset H_n$$
Observe that $$\{\nu_{i,j_1}\nu_{i,j_2}^{-1}\mid 1\leq i,j_1,j_2\leq n\}\subset \langle X\rangle $$ 
since $\nu_{i,j_1}\nu_{i,j_2}^{-1}=(\nu_{i,j_2}^{-1} f_i)(\nu_{i,j_1}^{-1}f_i)^{-1}$.
Also, observe that $(f_i\nu_{i,j}^{-1})\in \langle X\rangle$
since
$$(\nu_{i,i}^{-1}\nu_{i,j})(\nu_{i,j}^{-1} f_i)(\nu_{i,i}\nu_{i,j}^{-1})=(\nu_{i,i}^{-1}f_i)(\nu_{i,i}\nu_{i,j}^{-1})=(f_i\nu_{i,i}^{-1})(\nu_{i,i}\nu_{i,j}^{-1})=(f_i\nu_{i,j}^{-1})$$

An element of the form $f\nu_{i,j}^{-1}\in \langle X\rangle$ is called a \emph{special element} if $Supp(f)\cap Supp(\nu_{i,j}^{-1})=\emptyset$ and $Supp(f)$ is $\mathcal{I}$-small.
\end{defn}

\begin{lem}\label{special}
Let $f\nu_{i,j}^{-1}\in \langle X \rangle$ be a special element and let $g\in \{f_1^{\pm 1},...,f_n^{\pm 1}\}$
be such that $Supp(f^g)$ is $\mathcal{I}$-small.
Then there is a $1\leq k\leq n$ such that $f^g \nu_{i,k}^{-1}$ is also a special element of $\langle X\rangle$.
\end{lem}

\begin{proof}
Assume that $g=f_l$.
The proof where $g=f_l^{-1}$ is similar.
If $Supp(f)\cap Supp(f_l)=\emptyset$, then $f^g\nu_{i,j}^{-1}=f\nu_{i,j}^{-1}$ and we are done.
If $Supp(f)\cap Supp(f_l)\neq \emptyset$, then we consider the special element $$f\nu_{i,l}^{-1}=(f\nu_{i,j}^{-1})(\nu_{i,j}\nu_{i,l}^{-1})\in \langle X\rangle$$
Note that this is a special element since $Supp(f)$ is $\mathcal{I}$-small, and $Supp(f)\cap Supp(f_l)\neq \emptyset$, hence $Supp(f)\cap Supp(\nu_{i,l}^{-1})=\emptyset$.
It follows that $$(f\nu_{i,l}^{-1})^{f_l\nu_{l,l}^{-1}}=f^{f_l\nu_{l,l}^{-1}}\nu_{i,l}^{-1}=(f^{\nu_{l,l}^{-1}})^{f_l}\nu_{i,l}^{-1}$$
$$=f^{f_l}\nu_{i,l}^{-1}\in \langle X\rangle$$
is a special element.
\end{proof}

\begin{lem}\label{specialelements}
Let $J\subset \mathbf{S}^1$ be an open interval such that $J^c$ has nonempty interior.
For each $1\leq i\leq n, s\in \{\pm 1\}$, there exists an element $\gamma\in \langle X\rangle$ such that $\gamma\restriction J=f_i^{s}\restriction J$.
\end{lem}

\begin{proof}
We fix $i\in \{1,...,n\}$ and $s=-1$ (the proof for $s=+1$ is similar).
Consider the element $\nu_{i,j}$ for some $1\leq j\leq n$.
We know that $Supp(\nu_{i,j})\subset J_k$ for some $1\leq k\leq n$.
Consider the element $\nu_{i,j}\nu_{i,k}^{-1}\in \langle X\rangle$.

By minimality of the action of $G_n$ on $\mathbf{S}^1$, we can find an element $g=g_1...g_m$ for $g_i\in \{f_1,^{\pm 1},...,f_n^{\pm 1}\}$
such that $inf(J_k)\cdot g\subset J^c$.
By continuity, we find a $\mathcal{I}$-small open interval $I$ containing $inf(J_k)$ such that:
\begin{enumerate}
\item $I\cdot g\subset J^c$.
\item For each $1\leq s\leq m$, the interval $I\cdot g_1...g_s$ is $\mathcal{I}$-small.
\end{enumerate}
We let $h=f_k\nu_{k,k}^{-1}\in \langle X\rangle$.
It follows that for some large $l\in \mathbf{N}$ the element 
$$(\nu_{i,j}\nu_{i,k}^{-1})^{h^l}=\nu_{i,j}^{f_k^l}\nu_{i,k}^{-1}\in \langle X\rangle$$ 
has the property that if $\gamma=\nu_{i,j}^{f_k^l}$, then $Supp(\gamma)\subset I$ and $\gamma \nu_{i,k}^{-1}$ is a special element.
Note that this means that for each $1\leq s\leq m$, $Supp(\gamma)\cdot g_1...g_s$ is $\mathcal{I}$-small,
hence $Supp(\gamma^{g_1...g_s})$ is $\mathcal{I}$-small.

Applying Lemma \ref{special} to $\gamma\nu_{i,k}^{-1}$, we conclude that there is a $1\leq k_1\leq n$ such that
$\gamma^{g_1}\nu_{i,k_1}^{-1}\in \langle X\rangle$
is a special element.
Proceeding inductively, applying Lemma \ref{special} each time, we find $1\leq k_1,...,k_m\leq n$ such that 
$$\gamma^{g_1...g_s}\nu_{i,k_s}^{-1}\in \langle X\rangle\qquad \text{ is a special element for } 1\leq s\leq m$$
In particular, $\gamma^g\nu_{i,k_m}^{-1}\in \langle X\rangle$
and $(\gamma^g\nu_{i,k_m}^{-1})(\nu_{i,k_m} f_i^{-1})=\gamma^gf_i^{-1}\in \langle X\rangle$ is an element which satisfies the conclusion of the Lemma, since $Supp(\gamma^g)\subset J^c$.

\end{proof}

Now we prove our main theorem.

\begin{thm}
$H_n=\langle X\rangle$, and hence it is finitely generated.
\end{thm}

\begin{proof}
Let $f\in \langle X\rangle$ be any nontrivial element such that $Supp(f)^c$ has nonempty interior.
Since $H_n$ is simple, $\langle \langle f\rangle \rangle_{H_n}=H_n$.
To prove the claim, it suffices to show that for any element $g\in H_n$,
we have that $(f^{\pm 1})^g\in \langle X\rangle$.
We proceed by induction on the word length of $g$.
The base case is trivial.
Now assume that $g=hf_l$ for some $1\leq l\leq n$,
and that by the induction hypothesis $f^h\in \langle X\rangle$.
The proof for the case $g=hf_l^{-1}$ shall be similar.

Let $J=Supp(f^h)$.
Note that $J^c$ has nonempty interior.
Applying Lemma \ref{specialelements}, there exists an element $\gamma\in \langle X\rangle$ such that $\gamma\restriction J=f_l$.
We obtain that $f^{hf_l}=f^{h\gamma}\in \langle X\rangle$.
\end{proof}

We finish the proof of Theorem \ref{main2} by proving the following.

\begin{prop}
$G_n$ is left orderable.
\end{prop}

\begin{proof}
Recall that a group is left orderable if and only if it admits a faithful action by orientation preserving homeomorphisms on the real line.
Recall that $\textup{Homeo}^+(\mathbf{S}^1)$ admits a lift to $\mathbf{R}$ as follows:
$$1\to \mathbf{Z}\to \widetilde{\textup{Homeo}^+}(\mathbf{S}^1)\to \textup{Homeo}^+(\mathbf{S}^1)\to 1$$
where $\mathbf{Z}=\langle t\to t+n\mid n\in \mathbf{Z}\rangle$ and $\widetilde{\textup{Homeo}^+}(\mathbf{S}^1)$ is the centralizer of
$\langle t\to t+n\mid n\in \mathbf{Z}\rangle$ in $\textup{Homeo}^+(\mathbf{R})$.

We claim that the lift $\widetilde{G_n}$ of $G_n$ to a subgroup of $\textup{Homeo}^+(\mathbf{R})$ is isomorphic to $G_n$.
Denote the lifts of the generators $f_i$ as $h_i$, for each $1\leq i\leq n$.
It suffices to show that $$\widetilde{G_n}\cap \langle t\to t+n\mid n\in \mathbf{Z}\rangle=\{id_{\mathbf{R}}\}$$
If this were not the case, then we can find a word $$g=g_1...g_m\qquad g_i\in \{h_1,...,h_n\}$$
with the property that $x\cdot g=x+n$ for $n\in \mathbf{Z}\setminus \{0\}$.
The corresponding word $\gamma=\gamma_1...\gamma_m$, where each $h_i^{\pm 1}$ is replaced by $f_i^{\pm 1}$, must satisfy the following condition.
For any point in $x\in \mathbf{S}^1$, the sequence $$x, x\cdot \gamma_1,x\cdot \gamma_1\gamma_2,...,x\cdot \gamma_1...\gamma_m=x$$
must "go around the circle" a nontrivial number of times to arrive back to itself.
We will show that this is impossible, i.e. the above can only happen with "backtracking".

Let $x=inf(J_1)$. Consider the sequence 
$$x, x\cdot \gamma_1,x\cdot \gamma_1\gamma_2,...,x\cdot \gamma_1...\gamma_m=x$$
If $$x\cdot \gamma_1...\gamma_i=x\cdot \gamma_1...\gamma_{i+1}$$
we delete the occurrence of $\gamma_{i+1}$ from the word $\gamma_1...\gamma_m$, and adjust the indices
(replacing $j$ by $j-1$ for $j>i+1$) to obtain a new word $\gamma_1...\gamma_{m-1}$.
Whenever we find \emph{backtracking}, i.e. $$x\cdot \gamma_1...\gamma_{i-1}=x\cdot \gamma_1...\gamma_{i+1}$$
in $\gamma_1...\gamma_m$,
we remove $\gamma_i\gamma_{i+1}$ from the word, and adjust the indices (replacing $j$ by $j-2$ for $j>i+1$) to obtain a new word $\gamma_1...\gamma_{m-2}$.
At the end of this process, we obtain a new word $\gamma_1...\gamma_k$ such that
$$x, x\cdot \gamma_1,x\cdot \gamma_1\gamma_2,...,x\cdot \gamma_1...\gamma_k=x$$
such that for each $1\leq i<k$ $$x\cdot \gamma_1...\gamma_i\neq x\cdot \gamma_1...\gamma_{i+1}$$
and for each $1\leq i<k-1$ $$x\cdot \gamma_1...\gamma_i\neq x\cdot \gamma_1...\gamma_{i+2}$$

In this situation, an elementary inductive argument using condition $(*)$ implies that 
$$x\cdot \gamma_1...\gamma_k\notin \{inf(J_i),sup(J_i)\mid 1\leq i\leq n\}$$
contradicting our assumption that $x=inf(J_1)$.



\end{proof}

\end{document}